\documentclass[12pt,reqno]{amsproc}

\usepackage{amsmath,amsthm,amssymb,wasysym,mathtools}
\usepackage[inline]{enumitem}
\usepackage{caption}
\usepackage{subcaption}
\usepackage{graphicx}

\numberwithin{equation}{section}
\theoremstyle{plain}

\newtheorem{theorem}{Theorem}[section]

\newtheorem{remark}[theorem]{Remark}

\theoremstyle{definition}

\allowdisplaybreaks

\begin{document}
\title[Criteria for Starlikeness Using Schwarzian Derivatives]{Criteria for Starlikeness Using Schwarzian Derivatives }

\author[A. Sebastian]{Asha Sebastian}
\address{Department of Mathematics \\National Institute of Technology\\Tiruchirappalli-620015,  India }
\email{ashanitt18@gmail.com}

\author{V. Ravichandran}
\address{Department of Mathematics \\National Institute of Technology\\Tiruchirappalli-620015,  India }
\email{vravi68@gmail.com; ravic@nitt.edu}

\begin{abstract}
For a normalised analytic function  $f$ defined on the open unit disk in the complex plane, we determine several sufficient conditions for starlikeness  in terms of the quotients $Q_{ST}:=zf'(z)/f(z)$, $Q_{CV}:=1+zf''(z)/f'(z)$ and the Schwarzian derivative $Q_{SD}:=z^2\big(\left(f''(z)/f'(z)\right)'-\left(f''(z)/f'(z)\right)^2/2\big)$. These conditions were obtained by using the admissibility criteria of starlikeness in the theory of  second order differential subordination.
\end{abstract}

\dedicatory{Dedicated to Professor Milutin  Obradovi\'c }

\thanks{The first author is supported by  an  institute  fellowship from NIT Tiruchirappalli.}

\maketitle
\section{Introduction}A function $f:\mathbb{D}:=\{z\in\mathbb{C}:|z|<1\}\to\mathbb{C}$ is starlike if $t f(z)\in  f(\mathbb{D})$ for all $z\in \mathbb{D}$ and  $t\in [0,1]$.
We shall restrict our functions to belong to the  class  $\mathcal{A}$  of all analytic functions $f:\mathbb{D}\to \mathbb{C}$ normalized by  the condition $f(0) = f'(0)-1 = 0$. Let $\mathcal{S}\subset \mathcal{A}$ consists of univalent functions and $\mathcal{S}^* \subset \mathcal{A}$ be the class of starlike functions. A function $f\in\mathcal{A}$ is convex if $f(\mathbb{D})$ is convex and the class of all convex functions is denoted by $\mathcal{K}$. Analytically,  starlike and convex functions are characterized by
$\operatorname{Re} Q_{ST}>0$ and   $\operatorname{Re} Q_{CV}>0$ where $Q_{ST}:=zf'(z)/f(z)$ and $Q_{CV}:=1+zf''(z)/f'(z)$. For $0\leq \gamma <1$, the class $\mathcal{S}^*(\gamma)$  of  starlike functions of order $\gamma$ is defined by $\mathcal{S}^*(\gamma):=\{f\in \mathcal{A}: \operatorname{Re} Q_{ST}> \gamma\}$ and the class  $\mathcal{K}(\gamma)$  of convex functions of order $\gamma$ is defined by $\mathcal{K}:=\{f\in \mathcal{A}:  \operatorname{Re} Q_{CV}>\gamma\}$. The functions
in the classes $\mathcal{S}^*$ and $\mathcal{K}$ are univalent. A well-known univalence criteria of Nehari involves the Schwarzian derivative of function $f \in \mathcal{A}$   defined by $\{f,z\}:=\big(\left(f''(z)/f'(z)\right)'-\left(f''(z)/f'(z)\right)^2/2\big)$ and $Q_{SD}:=z^2\{f,z\}$ . Nehari \cite{MR0029999,MR0064145} studied necessary and sufficient conditions relating Schwarzian derivatives to univalency of functions $f\in \mathcal{A}$.   Schwarzian derivatives were studied by several authors (see \cite{ MR2557114,MR3608483}). Sharma et al.\ \cite{MR4308874} discussed sufficient conditions for strong starlike functions and Cho et al.  \cite{cho} studied higher order Schwarzian derivatives for Janowski classes.

Obradovi\'{c} \cite{MR1611761} has shown that the condition $|f''(z)|<1$ implies starlikeness of $f\in\mathcal{A}$ and the condition $|f''(z)|<1/2$ implies convexity. These simple conditions were further studied, among others, by  Tuneski \cite{MR1752408, MR1826485, MR2514893} and Kown and Sim \cite{MR4017391}. Our interest is to provide such simple sufficient conditions for starlikeness using $Q_{ST}$, $Q_{CV}$ and $Q_{SD}$. Our main tool in getting these result is the general theory of differential subordination introduced by Miller and Mocanu \cite{MR0783572}. 
Miller and Mocanu \cite[pp.244]{MR1760285} discussed on admissibility conditions related to the starlikeness and convexity; they proved that if $f \in \mathcal{A}$, with $f(z)f'(z)/z \neq 0$, $\operatorname{Re}\psi\left(Q_{ST}, Q_{CV}, Q_{SD}\right)>0$ then the function $f$ is starlike, provided the  function $\psi: \mathbb{C}^3 \rightarrow \mathbb{C}$ satisfy $\operatorname{Re}\psi(i\rho,i\tau,\xi+i\eta)\leqslant 0$, whenever $\rho, \tau, \xi, \eta\in \mathbb{R},\ \rho\tau\geqslant(1+3\rho^2)/2$ and $\rho\eta\geqslant0$. Several authors (see, for example, \cite{MR1078406, MR0399438, MR1894627, MR1408788, MR1340763, MR2179557, MR2162214}) applied this theory to investigate  criteria for the functions to be starlike or convex.    Ravichandran et al.\ \cite{MR1966516} proved that if $f \in \mathcal{A}$ satisfies $\operatorname{Re}Q_{ST}Q_{CV}>(\gamma+1)(\gamma-1/2)$, then the function $f$ is starlike of order $\gamma$. Motivated essentially by these works, we have a systematic discussion on various criteria involving $Q_{ST}$, $Q_{CV}$ and $Q_{SD}$   for the starlikeness of functions in class $\mathcal{A}$.


For a given set  $\Omega \subset \mathbb{C}$,  the class $\Psi(\Omega)$ of admissible functions consists of functions $\psi: \mathbb{C}^3 \times \mathbb{D} \rightarrow \mathbb{C}$ satisfying  the admissibility condition
\begin{equation}\label{sdlemma1}
\psi(i\rho,i\tau,\xi+i\eta) \notin \Omega
\end{equation}
for $z \in \mathbb{D}$, and for all real $\rho,\tau,\xi, \eta$ with
\begin{equation}\label{sdlemmacond}
\rho \tau\geqslant \frac{1}{2}(1+3\rho^2),\ \rho\eta\geqslant 0.
\end{equation}
Our theorems are proved by making use of  the following extension of the criteria  of Miller and Mocanu for the starlikeness of functions $f \in \mathcal{A}$  given in terms of the Schwarzian derivatives:

\begin{theorem}\cite[p. 9]{MR2470182} \label{sdlemma}If $f \in \mathcal{A}$ with  $f(z)f'(z)/z \neq 0$ satisfies  \[\psi(Q_{ST}, Q_{CV}, Q_{SD}) \in \Omega\] for some $\psi\in\Psi(\Omega)$, then the function $f$ is starlike.
\end{theorem}

\section{Criteria for starlikeness}
Lewandowski et al.\ \cite{MR0399438} discussed the criterion for starlikeness of a function $f \in \mathcal{A}$. Many authors have developed sufficient conditions for starlikeness and convexity of functions. See \cite{MR1078406, MR1894627, MR1408788, MR1340763, MR2179557, MR2162214}. In this section, we derive results relating the Schwarzian derivatives and starlikeness of functions in the class $\mathcal{A}$.
\begin{theorem}\label{sdtheorem1}
	Let $\alpha \geqslant0$ and $\beta \geqslant 0$. If the function $f \in \mathcal{A}$ satisfy any of the following inequalities
	\begin{enumerate}
		\item[(i)] $\operatorname{Re}\big(Q_{ST}(\alpha Q_{CV}+ \beta Q_{SD})\big)>-\alpha/2$,
		\item[(ii)] $\operatorname{Re}\big(Q_{CV}(\alpha Q_{ST}+ \beta Q_{SD})\big)>-\alpha/2$,
		\item[(iii)] $\operatorname{Re}\big(Q_{ST}(\alpha Q_{ST}+ \beta Q_{SD})\big)>0$,
		\item[(iv)] $\operatorname{Re}\big(Q_{CV}(\alpha Q_{CV}+ \beta Q_{SD})\big)>0$,
	\end{enumerate}
	then the function $f$ is starlike.
\end{theorem}
\begin{proof}
	For $i=1,2,3,4$, let $\Omega_i$ be defined by $\Omega_1:=\{w\in \mathbb{C}: \operatorname{Re}w>-\alpha/2\}=:\Omega_2$ and $\Omega_3:=\{w \in \mathbb{C}:\operatorname{Re}w>0\}=:\Omega_4$ and the functions $\psi_i:\mathbb{C}^3\to\mathbb{C}$ be defined  by 
	\begin{align*}
	\psi_1(u,v,w)=&u(\alpha v+\beta w),\\
	\psi_2(u,v,w)=&v(\alpha u+\beta w),\\
	\psi_3(u,v,w)=&u(\alpha u+\beta w)
	\shortintertext{and}\quad 
	\psi_4(v,w)=&v(\alpha v+\beta w).
	\end{align*}
	The hypothesis of the theorem shows that the function $\psi_i$ satisfies  
	\[\psi_i(Q_{ST},Q_{CV},Q_{SD})\in \Omega_i\quad \text{for\ } i=1,2,3,4.\] 
	It then follows from Theorem \ref{sdlemma}  that the function $f$ is starlike provided $\psi_i\in \Psi(\Omega_i)$.
	We complete the proof by showing  that the function $\psi_i\in \Psi(\Omega_i)$.
	
	Let $\alpha \geqslant0$ and $\beta \geqslant 0$ and  $\rho,\tau,\xi,\eta \in \mathbb{R}$ satisfy the conditions $\rho\tau\geqslant(1+3\rho^2)/2$ and $\rho\eta\geqslant0$. Then,  we have
	\begin{align*}
	\operatorname{Re}\psi_1(i\rho,i\tau,\xi+i\eta)=-\alpha\rho\tau-\beta\rho\eta
	\leqslant -\alpha\frac{(1+3\rho^2)}{2}\leqslant-\frac{\alpha}{2} 
	\end{align*}
	and this proves that  the function $\psi_1\in \Psi(\Omega_1)$. For the function $\psi_2$, we have
	\begin{align*}
	\operatorname{Re}\psi_2(i\rho,i\tau,\xi+i\eta)=-\alpha\rho\tau-\beta\tau\eta
	\leqslant -\alpha\frac{(1+3\rho^2)}{2}\leqslant-\frac{\alpha}{2},
	\end{align*}and so  the function $\psi_2 \in \Psi(\Omega_2)$.
	Similarly, we have 
	\begin{align*}
	\operatorname{Re}\psi_3(i\rho,i\tau,\xi+i\eta)=-\alpha\rho^2-\beta\rho\eta \leqslant 0,
	\shortintertext{and }
	\operatorname{Re}\psi_4(i\rho,i\tau, \xi+i\eta)=-\alpha\tau^2-\beta\tau\eta \leqslant 0,
	\end{align*} so that  the functions $\psi_3\in \Psi(\Omega_3)$  and  $\psi_4   \in \Psi(\Omega_4)$.
\end{proof}

\begin{remark}Theorem \ref{sdtheorem1} (i) with  $\alpha=1$, $\beta=0$ reduces to a  sufficient condition for starlikeness obtained by Ravichandran et al.\ \cite{MR1966516}.
\end{remark}

Various authors \cite{MR1356393, MR1639004, MR0783572}have investigated on expressions involving the product of the terms $Q_{ST}$ and $Q_{CV}$ for the study of starlikeness of functions.  The following theorems discuss the influence of Schwarzian derivatives in many such cases.
\begin{theorem}
	Let $\alpha \geqslant0$ and $\beta \geqslant 0$. If the function $f \in \mathcal{A}$ satisfy any of the following inequalities
	\begin{enumerate}
		\item[(i)] $\operatorname{Re}\big(Q_{CV}(\alpha Q_{ST}+(1-\alpha)Q_{ST}^2+ \beta Q_{SD})\big)>-\alpha/2$,
		\item[(ii)] $\operatorname{Re}\big(Q_{ST}(\alpha Q_{CV}+(1-\alpha)Q_{CV}^2+ \beta Q_{SD})\big)>-\alpha/2$,
		\item[(iii)] $\operatorname{Re}\big(Q_{ST}(\alpha Q_{CV}+(1-\alpha)Q_{ST}^2+ \beta Q_{SD})\big)>-\alpha/2$,
		\item[(iv)] $\operatorname{Re}\big(Q_{CV}(\alpha Q_{ST}+(1-\alpha)Q_{CV}^2+ \beta Q_{SD})\big)>-\alpha/2$,
	\end{enumerate}
	then the function $f$ is starlike.
\end{theorem}
\begin{proof}
	Let $\Omega$ be defined by $\Omega:=\{w\in \mathbb{C}: \operatorname{Re}w>-\alpha/2\}$ and for $i=1,2,3,4$, let the functions $\psi_i:\mathbb{C}^3\to\mathbb{C}$ be defined  by 
	\begin{align*}
	\psi_1(u,v,w)=&v(\alpha u+(1-\alpha)u^2+\beta w),\\
	\psi_2(u,v,w)=&u(\alpha v+(1-\alpha)v^2+\beta w),\\
	\psi_3(u,v,w)=&u(\alpha v+(1-\alpha)u^2+\beta w)\\ 
	\shortintertext{and}\quad 
	\psi_4(v,w)=&v(\alpha u+(1-\alpha)v^2+\beta w).
	\end{align*}
	The hypothesis of the theorem shows that the function $\psi_i$ satisfies  
	\[\psi_i(Q_{ST},Q_{CV},Q_{SD})\in \Omega\quad \text{for\ } i=1,2,3,4. \]
	It then follows from Theorem \ref{sdlemma}  that the function $f$ is starlike provided $\psi_i\in \Psi(\Omega)$.
	We complete the proof by showing  that the function $\psi_i\in \Psi(\Omega_i)$.
	
	Let $\alpha \geqslant0$ and $\beta \geqslant 0$ and  $\rho,\tau,\xi,\eta \in \mathbb{R}$ satisfy the conditions $\rho\tau\geqslant(1+3\rho^2)/2$ and $\rho\eta\geqslant0$. Then,  we have
	\begin{align*}
	\operatorname{Re}\psi_1(i\rho,i\tau,\xi+i\eta)=-\alpha\rho\tau-\beta\tau\eta
	\leqslant -\frac{\alpha}{2}(1+3\rho^2)\leqslant-\frac{\alpha}{2},
	\end{align*}
	and this proves that  the function $\psi_1\in \Psi(\Omega)$. For the function $\psi_2$, we have
	\begin{align*}
	\operatorname{Re}\psi_2(i\rho,i\tau,\xi+i\eta)=-\alpha\rho\tau-\beta\rho\eta,
	\leqslant -\alpha\frac{(1+3\rho^2)}{2}\leqslant-\frac{\alpha}{2},
	\end{align*}and so  the function $\psi_2 \in \Psi(\Omega_2)$.
	Similarly, we have 
	\begin{align*}
	\operatorname{Re}\psi_3(i\rho,i\tau,\xi+i\eta)=-\alpha\rho\tau-\beta\rho\eta\leqslant-\frac{\alpha}{2}(1+3\rho^2) \leqslant-\frac{\alpha}{2},
	\shortintertext{and }
	\operatorname{Re}\psi_4(i\rho,i\tau, \xi+i\eta)=-\alpha\rho\tau-\beta\tau\eta \leqslant-\frac{\alpha}{2}(1+3\rho^2)\leqslant -\frac{\alpha}{2},
	\end{align*} so that  the functions $\psi_3\in \Psi(\Omega)$  and  $\psi_4   \in \Psi(\Omega)$.
\end{proof}
\begin{theorem}\label{sdtheorem3}
	Let $\alpha \geqslant0$ and $\beta \geqslant 0$. If the function $f \in \mathcal{A}$ satisfy any of the following inequalities
	\begin{enumerate}
		\item[(i)] $\operatorname{Re}\big(Q_{ST}(\alpha Q_{ST}+(1-\alpha)Q_{ST}^2+ \beta Q_{SD})\big)>0$
		\item[(ii)] $\operatorname{Re}\big(Q_{CV}(\alpha Q_{CV}+(1-\alpha)Q_{CV}^2+ \beta Q_{SD})\big)>0$,
		\item[(iii)] $\operatorname{Re}\big(Q_{ST}(\alpha Q_{ST}+(1-\alpha)Q_{CV}^2+ \beta Q_{SD})\big)>0$,
		\item[(iv)] $\operatorname{Re}\big(Q_{CV}(\alpha Q_{CV}+(1-\alpha)Q_{ST}^2+ \beta Q_{SD})\big)>0$,
	\end{enumerate}
	then the function $f$ is starlike.
\end{theorem}
\begin{proof}
	Let $\Omega$ be defined by $\Omega:=\{w\in \mathbb{C}: \operatorname{Re}w>0\}$ and for $i=1,2,3,4$, let the functions $\psi_i:\mathbb{C}^3\to\mathbb{C}$ be defined  by 
	\begin{align*}
	\psi_1(u,v,w)=&u(\alpha u+(1-\alpha)u^2+\beta w),\\
	\psi_2(u,v,w)=&v(\alpha v+(1-\alpha)v^2+\beta w),\\
	\psi_3(u,v,w)=&u(\alpha u+(1-\alpha)v^2+\beta w)
	\shortintertext{and} 
	\psi_4(v,w)=&v(\alpha v+(1-\alpha)u^2+\beta w).
	\end{align*}
	The hypothesis of the theorem shows that the function $\psi_i$ satisfies  
	\[\psi_i(Q_{ST},Q_{CV},Q_{SD})\in \Omega\quad \text{for\ } i=1,2,3,4.\] 
	It then follows from Theorem \ref{sdlemma}  that the function $f$ is starlike provided $\psi_i\in \Psi(\Omega)$.
	We complete the proof by showing  that the function $\psi_i\in \Psi(\Omega)$.
	
	Let $\alpha \geqslant0$ and $\beta \geqslant 0$ and  $\rho,\tau,\xi,\eta \in \mathbb{R}$ satisfy the conditions $\rho\tau\geqslant(1+3\rho^2)/2$ and $\rho\eta\geqslant0$. Then,  we have
	\begin{align*}
	\operatorname{Re}\psi_1(i\rho,i\tau,\xi+i\eta)=-\alpha\rho^2-\beta\rho\eta
	\leqslant 0,
	\end{align*}
	and this proves that  the function $\psi_1\in \Psi(\Omega)$. For the function $\psi_2$, we have
	\begin{align*}
	\operatorname{Re}\psi_2(i\rho,i\tau,\xi+i\eta)=-\alpha\tau^2-\beta\tau\eta
	\leqslant 0,
	\end{align*}and so  the function $\psi_2 \in \Psi(\Omega_2)$.
	Similarly, we have 
	\begin{align*}
	\operatorname{Re}\psi_3(i\rho,i\tau,\xi+i\eta)=-\alpha\rho^2-\beta\rho\eta\leqslant 0,
	\shortintertext{and }
	\operatorname{Re}\psi_4(i\rho,i\tau, \xi+i\eta)=-\alpha\tau^2-\beta\tau\eta \leqslant 0,
	\end{align*} so that  the functions $\psi_3\in \Psi(\Omega)$  and  $\psi_4   \in \Psi(\Omega)$.
\end{proof}
\begin{remark}
	For $\alpha=1$ and $\beta=1$ in Part(iv) of Theorem \ref{sdtheorem3}, the obtained result is same as the one discussed by Miller and Mocanu \cite[pp.247]{MR1760285} for the expression $\operatorname{Re}\big(Q_{CV}^2+Q_{SD}\big)>0$.
\end{remark}
Using the theory of differential subordination, Owa\ and\  Obradovi\'{c} \cite{MR1071050} proved that the function $f \in \mathcal{A}$ is starlike, if $\operatorname{Re}\big((Q_{CV}^2/2)+Q_{SD}\big)>0$. In a generalised manner, we prove for some other cases as well.
\begin{theorem}
	Let $\alpha>0$ and $\beta \geqslant 0$. If the function $f \in \mathcal{A}$ satisfy any of the following inequalities
	\begin{enumerate}
		\item[(i)] $\operatorname{Re}\big(\beta Q_{SD}Q_{ST}+\alpha\left(1+Q_{CV}\right)^2\big)>\alpha$,
		\item[(ii)] $\operatorname{Re}\big(\beta Q_{SD}Q_{ST}+\alpha\left(1+Q_{ST}\right)^2\big)>\alpha$,
		\item[(iii)] $\operatorname{Re}\big(\beta Q_{SD}Q_{CV}+\alpha\left(1+Q_{CV}\right)^2\big)>\alpha$,
		\item[(iv)] $\operatorname{Re}\big(\beta Q_{SD}Q_{CV}+\alpha\left(1+Q_{ST}\right)^2\big)>\alpha$,
	\end{enumerate}
	then the function $f$ is starlike.
\end{theorem}
\begin{proof}
	Let $\Omega$ be defined by $\Omega:=\{w\in \mathbb{C}: \operatorname{Re}w>\alpha\}$ and for $i=1,2,3,4$, let the functions $\psi_i:\mathbb{C}^3\to\mathbb{C}$ be defined  by 
	\begin{align*}
	\psi_1(u,v,w)=&\beta uw+ \alpha(1+v)^2, \quad 
	\psi_2(u,v,w)=\beta uw+ \alpha(1+u)^2,\\
	\psi_3(u,v,w)=&\beta vw+\alpha\left(1+v\right)^2\quad 
	\text{and}\quad 
	\psi_4(v,w)=\beta vw+\alpha(1+u)^2.
	\end{align*}
	The hypothesis of the theorem shows that the function $\psi_i$ satisfies  
	\[\psi_i(Q_{ST},Q_{CV},Q_{SD})\in \Omega\quad \text{for\ } i=1,2,3,4.\] 
	It then follows from Theorem \ref{sdlemma}  that the function $f$ is starlike provided $\psi_i\in \Psi(\Omega)$.
	We complete the proof by showing  that the function $\psi_i\in \Psi(\Omega)$.
	
	Let $\alpha \geqslant0$ and $\beta \geqslant 0$ and  $\rho,\tau,\xi,\eta \in \mathbb{R}$ satisfy the conditions $\rho\tau\geqslant(1+3\rho^2)/2$ and $\rho\eta\geqslant0$. Then,  we have
	\begin{align*}
	\operatorname{Re}\psi_1(i\rho,i\tau,\xi+i\eta)=-\beta\rho\eta-\alpha\tau^2+\alpha
	\leqslant \alpha,
	\end{align*}
	and this proves that  the function $\psi_1\in \Psi(\Omega)$. For the function $\psi_2$, we have
	\begin{align*}
	\operatorname{Re}\psi_2(i\rho,i\tau,\xi+i\eta)=-\alpha\rho^2-\beta\rho\eta+\alpha
	\leqslant \alpha,
	\end{align*}and so  the function $\psi_2 \in \Psi(\Omega_2)$.
	Similarly, we have 
	\begin{align*}
	\operatorname{Re}\psi_3(i\rho,i\tau,\xi+i\eta)=-\alpha\tau^2-\beta\tau\eta+\alpha\leqslant \alpha,
	\shortintertext{and }
	\operatorname{Re}\psi_4(i\rho,i\tau, \xi+i\eta)=-\alpha\rho^2-\beta\tau\eta+\alpha \leqslant \alpha,
	\end{align*} so that  the functions $\psi_3\in \Psi(\Omega)$  and  $\psi_4   \in \Psi(\Omega)$.
\end{proof}
Miller and Mocanu  \cite{MR1760285} determined sufficient conditions relating the starlikeness of functions in class $\mathcal{A}$ and Schwarzian derivatives. As an application to the discussion, they obtained that for parameters $\alpha$, $\beta$, the sufficient conditions $\operatorname{Re}\big(\alpha Q_{ST}+\beta Q_{CV}+ Q_{ST}Q_{SD}\big)>0$ and $\operatorname{Re}\big(Q_{ST}Q_{CV}+ Q_{ST}Q_{SD}\big)>-1/2$ imply starlikeness. The forthcoming theorems follow as a generalisation of above observations.
\begin{theorem}
	Let $\alpha>0$ and $\beta \geqslant 0$. If the function $f \in \mathcal{A}$ satisfy any of the following inequalities
	\begin{enumerate}
		\item[(i)] $\operatorname{Re}\big(\beta Q_{SD}Q_{ST}+\alpha Q_{CV}\left(1+Q_{ST}\right)\big)>-\alpha/2$,
		\item[(ii)] $\operatorname{Re}\big(\beta Q_{SD}Q_{CV}+\alpha Q_{ST}\left(1+Q_{CV}\right)\big)>-\alpha/2$,
		\item[(iii)] $\operatorname{Re}\big(\beta Q_{SD}Q_{CV}+\alpha Q_{CV}\left(1+Q_{ST}\right)\big)>-\alpha/2$,
		\item[(iv)] $\operatorname{Re}\big(\beta Q_{SD}Q_{ST}+\alpha Q_{ST}\left(1+Q_{CV}\right)\big)>-\alpha/2$,
	\end{enumerate}
	then the function $f$ is starlike.
\end{theorem}
\begin{proof}
	Let $\Omega$ be defined by $\Omega:=\{w\in \mathbb{C}: \operatorname{Re}w>-\alpha/2\}$ and for $i=1,2,3,4$, let the functions $\psi_i:\mathbb{C}^3\to\mathbb{C}$ be defined  by 
	\begin{align*}
	\psi_1(u,v,w)=&\beta uw+ \alpha v(1+u), \quad 
	\psi_2(u,v,w)=\beta vw+ \alpha u(1+v),\\
	\psi_3(u,v,w)=&\beta vw+\alpha v\left(1+u\right)\quad 
	\text{and}\quad 
	\psi_4(v,w)=\beta uw+\alpha u(1+v).
	\end{align*}
	The hypothesis of the theorem shows that the function $\psi_i$ satisfies  
	\[\psi_i(Q_{ST},Q_{CV},Q_{SD})\in \Omega\quad \text{for\ } i=1,2,3,4.\] 
	It then follows from Theorem \ref{sdlemma}  that the function $f$ is starlike provided $\psi_i\in \Psi(\Omega)$.
	We complete the proof by showing  that the function $\psi_i\in \Psi(\Omega)$.
	
	Let $\alpha \geqslant0$ and $\beta \geqslant 0$ and  $\rho,\tau,\xi,\eta \in \mathbb{R}$ satisfy the conditions $\rho\tau\geqslant(1+3\rho^2)/2$ and $\rho\eta\geqslant0$. Then,  we have
	\begin{align*}
	\operatorname{Re}\psi_1(i\rho,i\tau,\xi+i\eta)=-\beta\rho\eta-\alpha\rho\tau\leqslant-\frac{\alpha}{2}(1+3\rho^2)
	\leqslant -\frac{\alpha}{2},
	\end{align*}
	and this proves that  the function $\psi_1\in \Psi(\Omega)$. For the function $\psi_2$, we have
	\begin{align*}
	\operatorname{Re}\psi_2(i\rho,i\tau,\xi+i\eta)=-\alpha\rho\tau-\beta\tau\eta
	\leqslant -\frac{\alpha}{2}(1+3\rho^2)\leqslant-\frac{\alpha}{2},
	\end{align*}and so  the function $\psi_2 \in \Psi(\Omega_2)$.
	Similarly, we have 
	\begin{align*}
	\operatorname{Re}\psi_3(i\rho,i\tau,\xi+i\eta)=-\alpha\rho\tau-\beta\tau\eta\leqslant -\frac{\alpha}{2}(1+3\rho^2)\leqslant-\frac{\alpha}{2},
	\shortintertext{and }
	\operatorname{Re}\psi_4(i\rho,i\tau, \xi+i\eta)=-\alpha\rho\tau-\beta\rho\eta \leqslant-\frac{\alpha(1+3\rho^2)}{2}\leqslant-\frac{\alpha}{2},
	\end{align*} so that  the functions $\psi_3\in \Psi(\Omega)$  and  $\psi_4   \in \Psi(\Omega)$.
\end{proof}
\begin{theorem}
	Let $\alpha>0$ and $\beta \geqslant 0$. If the function $f \in \mathcal{A}$ satisfy any of the following inequalities
	\begin{enumerate}
		\item[(i)] $\operatorname{Re}\big(\beta Q_{SD}Q_{ST}+\alpha Q_{ST}\left(1+Q_{ST}\right)\big)>0$,
		\item[(ii)] $\operatorname{Re}\big(\beta Q_{SD}Q_{CV}+\alpha Q_{ST}\left(1+Q_{ST}\right)\big)>0$,
		\item[(iii)] $\operatorname{Re}\big(\beta Q_{SD}Q_{ST}+\alpha Q_{CV}\left(1+Q_{CV}\right)\big)>0$,
		\item[(iv)] $\operatorname{Re}\big(\beta Q_{SD}Q_{CV}+\alpha Q_{CV}\left(1+Q_{CV}\right)\big)>0$,
	\end{enumerate}
	then the function $f$ is starlike.
\end{theorem}
\begin{proof} 	Let $\Omega$ be defined by $\Omega:=\{w\in \mathbb{C}: \operatorname{Re}w>0\}$ and for $i=1,2,3,4$, let the functions $\psi_i:\mathbb{C}^3\to\mathbb{C}$ be defined  by 
	\begin{alignat*}{3}
	\psi_1(u,v,w)=&\beta uw+ \alpha u(1+u), \quad 
	\psi_2(u,v,w)=\beta vw+ \alpha u(1+u),\\
	\psi_3(u,v,w)=&\beta uw+\alpha v\left(1+v\right)\quad 
	\text{and}\quad 
	\psi_4(v,w)=\beta vw+\alpha v(1+v).
	\end{alignat*}
	The hypothesis of the theorem shows that the function $\psi_i$ satisfies  
	\[\psi_i(Q_{ST},Q_{CV},Q_{SD})\in \Omega\quad \text{for\ } i=1,2,3,4.\] 
	It then follows from Theorem \ref{sdlemma}  that the function $f$ is starlike provided $\psi_i\in \Psi(\Omega)$.
	We complete the proof by showing  that the function $\psi_i\in \Psi(\Omega)$.
	
	Let $\alpha>0$ and $\beta \geqslant 0$ and  $\rho,\tau,\xi,\eta \in \mathbb{R}$ satisfy the conditions $\rho\tau\geqslant(1+3\rho^2)/2$ and $\rho\eta\geqslant0$. Then,  we have
	\begin{align*}
	\operatorname{Re}\psi_1(i\rho,i\tau,\xi+i\eta)=-\alpha\rho^2-\beta\rho\eta
	\leqslant 0,
	\end{align*}
	and this proves that  the function $\psi_1\in \Psi(\Omega)$. For the function $\psi_2$, we have
	\begin{align*}
	\operatorname{Re}\psi_2(i\rho,i\tau,\xi+i\eta)=-\alpha\rho^2-\beta\tau\eta
	\leqslant 0,
	\end{align*}and so  the function $\psi_2 \in \Psi(\Omega_2)$.
	Similarly, we have 
	\begin{align*}
	\operatorname{Re}\psi_3(i\rho,i\tau,\xi+i\eta)=-\alpha\tau^2-\beta\rho\eta\leqslant 0,
	\shortintertext{and }
	\operatorname{Re}\psi_4(i\rho,i\tau, \xi+i\eta)=-\alpha\tau^2-\beta\tau\eta \leqslant 0,
	\end{align*} so that  the functions $\psi_3\in \Psi(\Omega)$  and  $\psi_4   \in \Psi(\Omega)$.
\end{proof}
Some authors \cite{MR1356393, MR1071050} considered the powers of the expressions $Q_{ST}$, $Q_{CV}$ and analysed their significance in the starlikeness of a function. The following theorems with Schwarzian derivatives are examined in a similar manner.

\begin{theorem}
	Let $0\leqslant\alpha\leqslant1$ and $\beta \geqslant 0$. If the function $f \in \mathcal{A}$ satisfy any of the following inequalities
	\begin{enumerate}
		\item[(i)] $\operatorname{Re}\big(\alpha Q_{ST}+(1-\alpha)Q_{ST}^2+\beta Q_{SD}Q_{CV}\big)>0$,
		\item[(ii)] $\operatorname{Re}\big(\alpha Q_{ST}+(1-\alpha)Q_{ST}^2+\beta Q_{SD}Q_{ST}\big)>0$,
		\item[(iii)] $\operatorname{Re}\big(\alpha Q_{ST}+(1-\alpha)Q_{CV}^2+\beta Q_{SD}Q_{ST}\big)>0$,
		\item[(iv)] $\operatorname{Re}\big(\alpha Q_{CV}+(1-\alpha)Q_{ST}^2+\beta Q_{SD}Q_{ST}\big)>0$,
		\item[(v)] $\operatorname{Re}\big(\alpha Q_{CV}+(1-\alpha)Q_{CV}^2+\beta Q_{SD}Q_{CV}\big)>0$,
		\item[(vi)] $\operatorname{Re}\big(\alpha Q_{CV}+(1-\alpha)Q_{ST}^2+\beta Q_{SD}Q_{CV}\big)>0$,
		\item[(vii)] $\operatorname{Re}\big(\alpha Q_{ST}+(1-\alpha)Q_{CV}^2+\beta Q_{SD}Q_{CV}\big)>0$,
		\item[(viii)] $\operatorname{Re}\big(\alpha Q_{CV}+(1-\alpha)Q_{CV}^2+\beta Q_{SD}Q_{ST}\big)>0$,
		
	\end{enumerate}
	then the function $f$ is starlike.
\end{theorem}
\begin{proof}
	Let $\Omega$ be defined by $\Omega:=\{w\in \mathbb{C}: \operatorname{Re}w>0\}$ and for $i=1,2,\cdots 8$, let the functions $\psi_i:\mathbb{C}^3\to\mathbb{C}$ be defined  by 
	\begin{align*}
	\psi_1(u,v,w)=&\alpha u+(1-\alpha)u^2+\beta vw,\\ 
	\psi_2(u,v,w)=&\alpha u+(1-\alpha)u^2+\beta uw,\\
	\psi_3(u,v,w)=&\alpha u+(1-\alpha)v^2+\beta uw,\\ 
	\psi_4(u,v,w)=&\alpha v+(1-\alpha)u^2+\beta uw,\\
	\psi_5(u,v,w)=&\alpha v+(1-\alpha)v^2+\beta vw,\\ 
	\psi_6(u,v,w)=&\alpha v+(1-\alpha)u^2+\beta vw,\\
	\psi_7(u,v,w)=&\alpha u+(1-\alpha)v^2+\beta vw
	\shortintertext{and}
	\psi_8(v,w)=&\alpha v+(1-\alpha)v^2+\beta uw.
	\end{align*}
	The hypothesis of the theorem shows that the function $\psi_i$ satisfies  
	\[\psi_i(Q_{ST},Q_{CV},Q_{SD})\in \Omega\quad \text{for\ } i=1,2,\cdots 8.\] 
	It then follows from Theorem \ref{sdlemma}  that the function $f$ is starlike provided $\psi_i\in \Psi(\Omega)$.
	We complete the proof by showing  that the function $\psi_i\in \Psi(\Omega)$.
	
	Let $0\leqslant\alpha\leqslant1$ and $\beta \geqslant 0$ and  $\rho,\tau,\xi,\eta \in \mathbb{R}$ satisfy the conditions $\rho\tau\geqslant(1+3\rho^2)/2$ and $\rho\eta\geqslant0$. Then,  we have
	\begin{align*}
	\operatorname{Re}\psi_1(i\rho,i\tau,\xi+i\eta)=-(1-\alpha)\rho^2-\beta \tau\eta
	\leqslant 0,
	\end{align*}
	and this proves that  the function $\psi_1\in \Psi(\Omega)$. For the function $\psi_2$, we have
	\begin{align*}
	\operatorname{Re}\psi_2(i\rho,i\tau,\xi+i\eta)=-(1-\alpha)\rho^2-\beta\rho\eta
	\leqslant 0,
	\end{align*}and so  the function $\psi_2 \in \Psi(\Omega_2)$.
	Similarly, we have 
	\begin{align*}
	\operatorname{Re}\psi_3(i\rho,i\tau,\xi+i\eta)=-(1-\alpha)\tau^2-\beta\rho\eta\leqslant 0,
	\shortintertext{and }
	\operatorname{Re}\psi_4(i\rho,i\tau, \xi+i\eta)=-(1-\alpha)\rho^2-\beta\rho\eta \leqslant 0,
	\end{align*} so that  the functions $\psi_3\in \Psi(\Omega)$  and  $\psi_4   \in \Psi(\Omega)$. Proceeding in a similar way, we have
	\begin{align*}
	\operatorname{Re}\psi_5(i\rho,i\tau,\xi+i\eta)=-(1-\alpha)\tau^2-\beta \tau \eta
	\leqslant 0,
	\end{align*}
	and this proves that  the function $\psi_5\in \Psi(\Omega)$. For the function $\psi_6$, we have
	\begin{align*}
	\operatorname{Re}\psi_2(i\rho,i\tau,\xi+i\eta)=-(1-\alpha)\rho^2-\beta\tau\eta
	\leqslant 0,
	\end{align*}and so  the function $\psi_6 \in \Psi(\Omega_2)$.
	Similarly, we have 
	\begin{align*}
	\operatorname{Re}\psi_7(i\rho,i\tau,\xi+i\eta)=-(1-\alpha)\tau^2-\beta\tau\eta\leqslant 0,
	\shortintertext{and }
	\operatorname{Re}\psi_8(i\rho,i\tau, \xi+i\eta)=-(1-\alpha)\tau^2-\beta\rho\eta \leqslant 0,
	\end{align*} so that  the functions $\psi_7\in \Psi(\Omega)$  and  $\psi_8   \in \Psi(\Omega)$.

\end{proof}
\begin{theorem}\label{sdtheorem8}
	Let $0\leqslant\alpha\leqslant1$ and $\beta \geqslant 0$. If any of the following two inequalities hold for the function $f \in \mathcal{A}$,
	\begin{enumerate}
		
		\item[(i)] $\operatorname{Re}\big(Q_{CV}(\alpha Q_{ST}+(1-\alpha)Q_{CV}+\beta Q_{SD})\big)>-\alpha/2$,
		\item[(ii)] $\operatorname{Re}\big(Q_{ST}(\alpha Q_{CV}+(1-\alpha)Q_{ST}+\beta Q_{SD})\big)>-\alpha/2$,		
	\end{enumerate}
	then the function $f$ is starlike.
\end{theorem}
\begin{proof} 
	Let $\Omega$ be defined by $\Omega:=\{w\in \mathbb{C}: \operatorname{Re}w>-\alpha/2\}$ and for $i=1,2$, let the functions $\psi_i:\mathbb{C}^3\to\mathbb{C}$ be defined  by 
	\begin{alignat*}{3}
	\psi_1(u,v,w)=&u(\alpha u+(1-\alpha)v+\beta w) 
	\shortintertext{and}
	\psi_2(u,v,w)=&u(\alpha v+(1-\alpha)u+\beta w).
	\end{alignat*}
	The hypothesis of the theorem shows that the function $\psi_i$ satisfies  
	\[\psi_i(Q_{ST},Q_{CV},Q_{SD})\in \Omega\quad \text{for\ } i=1,2.\] 
	It then follows from Theorem \ref{sdlemma}  that the function $f$ is starlike provided $\psi_i\in \Psi(\Omega)$.
	We complete the proof by showing  that the function $\psi_i\in \Psi(\Omega)$.
	
	Let $0\leqslant\alpha\leqslant1$ and $\beta \geqslant 0$ and  $\rho,\tau,\xi,\eta \in \mathbb{R}$ satisfy the conditions $\rho\tau\geqslant(1+3\rho^2)/2$ and $\rho\eta\geqslant0$. Then,  we have
	\begin{align*}
	\operatorname{Re}\psi_1(i\rho,i\tau,\xi+i\eta)=-\alpha\rho\tau-(1-\alpha)\tau^2-\beta \tau\eta
	\leqslant -\frac{\alpha}{2}(1+3\rho^2)\leqslant-\frac{\alpha}{2},
	\end{align*}
	and this proves that  the function $\psi_1\in \Psi(\Omega)$. For the function $\psi_2$, we have
	\begin{align*}
	\operatorname{Re}\psi_2(i\rho,i\tau,\xi+i\eta)=-(1-\alpha)\rho^2-\alpha\rho\tau-\beta\tau\eta
	\leqslant -\frac{\alpha}{2}(1+3\rho^2)\leqslant-\frac{\alpha}{2},
	\end{align*}and so  the function $\psi_2 \in \Psi(\Omega_2)$.	
\end{proof}
\begin{remark}
	Let $\alpha=1$ and $\beta=0$, then the results obtained from both Part(i) and Part(ii) of Theorem \ref{sdtheorem8} is same as the sufficent condition for starlikeness obtained by Ramesha and Padmanabhan \cite{MR1340763}.
\end{remark}
\begin{theorem}\label{sdtheorem9}
	Let $\alpha> 0$ and $\beta \geqslant 0$. If the function $f \in \mathcal{A}$ satisfy any of the following inequalities
	\begin{enumerate}
		\item[(i)] $\operatorname{Re}\big(\alpha(Q_{CV}/Q_{ST})+\beta(Q_{SD}/Q_{ST})\big)<3\alpha/2$ ,
		\item[(ii)] $\operatorname{Re}\big(\alpha(Q_{CV}/Q_{ST})+\beta(Q_{SD}/Q_{CV})\big)<3\alpha/2$,	
	\end{enumerate}
	then the function $f$ is starlike.
\end{theorem}
\begin{proof}
	Let $\Omega$ be defined by $\Omega:=\{w\in \mathbb{C}: \operatorname{Re}w<3\alpha/2\}$ and for $i=1,2$, let the functions $\psi_i:\mathbb{C}^3\to\mathbb{C}$ be defined  by 
	\begin{align*}
	\psi_1(u,v,w)=\alpha(v/u)+\beta(w/u) \quad 
	\text{and}\quad
	\psi_2(u,v,w)=\alpha(v/u)+\beta(w/v).
	\end{align*}
	The hypothesis of the theorem shows that the function $\psi_i$ satisfies  
	\[\psi_i(Q_{ST},Q_{CV},Q_{SD})\in \Omega\quad \text{for\ } i=1,2.\] 
	It then follows from Theorem \ref{sdlemma}  that the function $f$ is starlike provided $\psi_i\in \Psi(\Omega)$.
	We complete the proof by showing  that the function $\psi_i\in \Psi(\Omega)$.
	
	Let $\alpha> 0$ and $\beta \geqslant 0$ and  $\rho,\tau,\xi,\eta \in \mathbb{R}$ satisfy the conditions $\rho\tau\geqslant(1+3\rho^2)/2$ and $\rho\eta\geqslant0$. Then,  we have
	\begin{align*}
	\operatorname{Re}\psi_1(i\rho,i\tau,\xi+i\eta)=\frac{\alpha\tau}{\rho}+\frac{\beta\eta}{\rho}
	\geqslant \frac{\alpha}{2}\frac{(1+3\rho^2)}{\rho^2}\geqslant\frac{3\alpha}{2},
	\end{align*}
	and this proves that  the function $\psi_1\in \Psi(\Omega)$. For the function $\psi_2$, we have
	\begin{align*}
	\operatorname{Re}\psi_2(i\rho,i\tau,\xi+i\eta)=\frac{\alpha\tau}{\rho}+\frac{\beta\eta}{\tau}
	\geqslant \frac{\alpha}{2}\frac{(1+3\rho^2)}{\rho^2}\geqslant\frac{3\alpha}{2},
	\end{align*}and so  the function $\psi_2 \in \Psi(\Omega_2)$.
\end{proof}
\begin{remark}
	The sufficient conditions for starlikeness obtained by Tuneski \cite[pp.523]{MR1761956} follows from Part(i) and Part(ii) of  Theorem \ref{sdtheorem9} when $\alpha=1$ and $\beta=0$.
\end{remark}
\begin{theorem}
	Let $\alpha> 0$ and $\beta \leqslant 0$. If the function $f \in \mathcal{A}$ satisfy any of the following inequalities
	\begin{enumerate}
		\item[(i)] $\operatorname{Re}\big(\alpha(Q_{ST}/Q_{CV})+\beta(Q_{SD}/Q_{CV})\big)>2\alpha/3$ ,
		\item[(ii)] $\operatorname{Re}\big(\alpha(Q_{ST}/Q_{CV})+\beta(Q_{SD}/Q_{ST})\big)>2\alpha/3$,	
	\end{enumerate}
	then the function $f$ is starlike.
\end{theorem}
\begin{proof}
	Let $\Omega$ be defined by $\Omega:=\{w\in \mathbb{C}: \operatorname{Re}w>2\alpha/3\}$ and for $i=1,2$, let the functions $\psi_i:\mathbb{C}^3\to\mathbb{C}$ be defined  by 
	\begin{align*}
	\psi_1(u,v,w)=\alpha (u/v)+\beta (w/v) \quad 
	\text{and}\quad
	\psi_2(u,v,w)=\alpha(u/v)+\beta(w/u).
	\end{align*}
	The hypothesis of the theorem shows that the function $\psi_i$ satisfies  
	\[\psi_i(Q_{ST},Q_{CV},Q_{SD})\in \Omega\quad \text{for\ } i=1,2.\] 
	It then follows from Theorem \ref{sdlemma}  that the function $f$ is starlike provided $\psi_i\in \Psi(\Omega)$.
	We complete the proof by showing  that the function $\psi_i\in \Psi(\Omega)$.
	
	Let $\alpha> 0$ and $\beta \leqslant 0$ and  $\rho,\tau,\xi,\eta \in \mathbb{R}$ satisfy the conditions $\rho\tau\geqslant(1+3\rho^2)/2$ and $\rho\eta\geqslant0$. Then,  we have
	\begin{align*}
	\operatorname{Re}\psi_1(i\rho,i\tau,\xi+i\eta)=\frac{\alpha\rho}{\tau}+ \frac{\beta\eta}{\tau}\leqslant \frac{2\alpha\rho^2}{(1+3\rho^2)}\leqslant\frac{2\alpha}{3},
	\end{align*}
	and this proves that  the function $\psi_1\in \Psi(\Omega)$. For the function $\psi_2$, we have
	\begin{align*}
	\operatorname{Re}\psi_2(i\rho,i\tau,\xi+i\eta)=\frac{\alpha\rho}{\tau}+\frac{\beta\eta}{\rho}\leqslant\frac{2\alpha\rho^2}{(1+3\rho^2)}\leqslant\frac{2\alpha}{3},
	\end{align*}and so  the function $\psi_2 \in \Psi(\Omega_2)$.
\end{proof}
\begin{theorem}
	Let $\alpha < 0$ and $\beta \in \mathbb{R}$. If the function $f \in \mathcal{A}$ satisfy any of the following inequalities
	\begin{enumerate}
		\item[(i)] $\operatorname{Re}\big(\alpha(Q_{SD}/Q_{ST})+\beta Q_{ST}\big)>0$ ,
		\item[(ii)] $\operatorname{Re}\big(\alpha(Q_{SD}/Q_{CV})+\beta Q_{ST}\big)>0$ ,	
		\item[(iii)] $\operatorname{Re}\big(\alpha(Q_{SD}/Q_{ST})+\beta Q_{CV}\big)>0$ ,
		\item[(iv)] $\operatorname{Re}\big(\alpha(Q_{SD}/Q_{CV})+\beta Q_{CV}\big)>0$ ,
	\end{enumerate}
	then the function $f$ is starlike.
\end{theorem}
\begin{proof}
	Let $\Omega$ be defined by $\Omega:=\{w\in \mathbb{C}: \operatorname{Re}w>0\}$ and for $i=1,2,3,4$, let the functions $\psi_i:\mathbb{C}^3\to\mathbb{C}$ be defined  by 
	\begin{align*}
	\psi_1(u,v,w)=&\alpha(w/u)+\beta u, \quad 
	\psi_2(u,v,w)=\alpha (w/v)+\beta u,\\
	\psi_3(u,v,w)=&\alpha (w/u)+\beta v\quad 
	\text{and}\quad 
	\psi_4(v,w)=\alpha (w/v)+\beta v.
	\end{align*}
	The hypothesis of the theorem shows that the function $\psi_i$ satisfies  
	\[\psi_i(Q_{ST},Q_{CV},Q_{SD})\in \Omega\quad \text{for\ } i=1,2,3,4.\] 
	It then follows from Theorem \ref{sdlemma}  that the function $f$ is starlike provided $\psi_i\in \Psi(\Omega)$.
	We complete the proof by showing  that the function $\psi_i\in \Psi(\Omega)$.
	
	Let $\alpha < 0$ and $\beta \in \mathbb{R}$ and  $\rho,\tau,\xi,\eta \in \mathbb{R}$ satisfy the conditions $\rho\tau\geqslant(1+3\rho^2)/2$ and $\rho\eta\geqslant0$. Then,  we have
	\begin{align*}
	\operatorname{Re}\psi_1(i\rho,i\tau,\xi+i\eta)=\frac{\alpha\eta}{\rho}\leqslant 0,
	\end{align*}
	and this proves that  the function $\psi_1\in \Psi(\Omega)$. For the function $\psi_2$, we have
	\begin{align*}
	\operatorname{Re}\psi_2(i\rho,i\tau,\xi+i\eta)=\frac{\alpha\eta}{\tau}
	\leqslant 0,
	\end{align*}and so  the function $\psi_2 \in \Psi(\Omega_2)$.
	Similarly, we have 
	\begin{align*}
	\operatorname{Re}\psi_3(i\rho,i\tau,\xi+i\eta)=\frac{\alpha\eta}{\rho}\leqslant 0,
	\shortintertext{and }
	\operatorname{Re}\psi_4(i\rho,i\tau, \xi+i\eta)=\frac{\alpha\eta}{\tau} \leqslant 0,
	\end{align*} so that  the functions $\psi_3\in \Psi(\Omega)$  and  $\psi_4   \in \Psi(\Omega)$.
	
\end{proof}
Many authors have dedicated significant part of their works on developing conditions for the functions to be starlike. In such a way, the quotients $Q_{CV}/Q_{ST}$, $(Q_{CV}-\gamma)/Q_{ST}$ were introduced and examined. See \cite{MR1684341,MR2097775, MR1761956, MR2675019}. Along with the quotients, we consider Schwarzian derivatives and discuss its consequences in the study of starlikeness of functions.

\begin{theorem}
	Let $\alpha \leqslant0$ and $\beta \geqslant0$. If the function $f \in \mathcal{A}$ satisfy any of the following inequalities
	\begin{enumerate}
		\item[(i)] $\operatorname{Re}\big(\alpha(Q_{SD}/Q_{ST})+Q_{ST}(1+\beta Q_{ST})\big)>0$ ,
		\item[(ii)] $\operatorname{Re}\big(\alpha(Q_{SD}/Q_{CV})+Q_{ST}(1+\beta Q_{ST})\big)>0$,	
		\item[(iii)] $\operatorname{Re}\big(\alpha(Q_{SD}/Q_{CV})+Q_{CV}(1+\beta Q_{CV})\big)>0$ ,
		\item[(iv)] $\operatorname{Re}\big(\alpha(Q_{SD}/Q_{ST})+Q_{CV}(1+\beta Q_{CV})\big)>0$ ,
	\end{enumerate}
	then the function $f$ is starlike.
\end{theorem}
\begin{proof}
	Let $\Omega$ be defined by $\Omega:=\{w\in \mathbb{C}: \operatorname{Re}w>0\}$ and for $i=1,2,3,4$, let the functions $\psi_i:\mathbb{C}^3\to\mathbb{C}$ be defined  by 
	\begin{align*}
	\psi_1(u,v,w)=&\alpha(w/u)+u(1+\beta u), \quad 
	\psi_2(u,v,w)=\alpha (w/v)+u(1+\beta u),\\
	\psi_3(u,v,w)=&\alpha (w/v)+v(1+\beta v)\quad 
	\text{and}\quad 
	\psi_4(v,w)=\alpha (w/u)+v(1+\beta v).
	\end{align*}
	The hypothesis of the theorem shows that the function $\psi_i$ satisfies  
	\[\psi_i(Q_{ST},Q_{CV},Q_{SD})\in \Omega\quad \text{for\ } i=1,2,3,4.\] 
	It then follows from Theorem \ref{sdlemma}  that the function $f$ is starlike provided $\psi_i\in \Psi(\Omega)$.
	We complete the proof by showing  that the function $\psi_i\in \Psi(\Omega)$.
	
	Let $\alpha \leqslant0$ and $\beta \geqslant0$ and  $\rho,\tau,\xi,\eta \in \mathbb{R}$ satisfy the conditions $\rho\tau\geqslant(1+3\rho^2)/2$ and $\rho\eta\geqslant0$. Then,  we have
	\begin{align*}
	\operatorname{Re}\psi_1(i\rho,i\tau,\xi+i\eta)=\frac{\alpha\eta}{\rho}-\beta\rho^2\leqslant 0,
	\end{align*}
	and this proves that  the function $\psi_1\in \Psi(\Omega)$. For the function $\psi_2$, we have
	\begin{align*}
	\operatorname{Re}\psi_2(i\rho,i\tau,\xi+i\eta)=\frac{\alpha\eta}{\tau}-\beta\rho^2
	\leqslant 0,
	\end{align*}and so  the function $\psi_2 \in \Psi(\Omega_2)$.
	Similarly, we have 
	\begin{align*}
	\operatorname{Re}\psi_3(i\rho,i\tau,\xi+i\eta)=\frac{\alpha\eta}{\tau}-\beta\tau^2\leqslant 0,
	\shortintertext{and }
	\operatorname{Re}\psi_4(i\rho,i\tau, \xi+i\eta)=\frac{\alpha\eta}{\rho}-\beta\tau^2 \leqslant 0,
	\end{align*} so that  the functions $\psi_3\in \Psi(\Omega)$  and  $\psi_4   \in \Psi(\Omega)$.

\end{proof}
\begin{theorem}
	Let $\alpha \leqslant0$ and $\beta \geqslant0$. If the function $f \in \mathcal{A}$ satisfy any of the following inequalities
	\begin{enumerate}
		\item[(i)] $\operatorname{Re}\big(\alpha(Q_{SD}/Q_{ST})+Q_{ST}(1+\beta Q_{CV})\big)>-\beta/2$ ,
		\item[(ii)] $\operatorname{Re}\big(\alpha(Q_{SD}/Q_{ST})+Q_{CV}(1+\beta Q_{ST})\big)>-\beta/2$,	
		\item[(iii)] $\operatorname{Re}\big(\alpha(Q_{SD}/Q_{CV})+Q_{CV}(1+\beta Q_{ST})\big)>-\beta/2$ ,
		\item[(iv)] $\operatorname{Re}\big(\alpha(Q_{SD}/Q_{CV})+Q_{ST}(1+\beta Q_{CV})\big)>-\beta/2$ ,
	\end{enumerate}
	then the function $f$ is starlike.
\end{theorem}
\begin{proof}
	Let $\Omega$ be defined by $\Omega:=\{w\in \mathbb{C}: \operatorname{Re}w>-\beta/2\}$ and for $i=1,2,3,4$, let the functions $\psi_i:\mathbb{C}^3\to\mathbb{C}$ be defined  by 
	\begin{align*}
	\psi_1(u,v,w)=&\alpha(w/u)+u(1+\beta v), \quad 
	\psi_2(u,v,w)=\alpha (w/u)+v(1+\beta u),\\
	\psi_3(u,v,w)=&\alpha (w/v)+v(1+\beta u)\quad 
	\text{and}\quad 
	\psi_4(v,w)=\alpha (w/v)+u(1+\beta v).
	\end{align*}
	The hypothesis of the theorem shows that the function $\psi_i$ satisfies  
	\[\psi_i(Q_{ST},Q_{CV},Q_{SD})\in \Omega\quad \text{for\ } i=1,2,3,4.\] 
	It then follows from Theorem \ref{sdlemma}  that the function $f$ is starlike provided $\psi_i\in \Psi(\Omega)$.
	We complete the proof by showing  that the function $\psi_i\in \Psi(\Omega)$.
	
	Let $\alpha \geqslant0$ and $\beta \geqslant 0$ and  $\rho,\tau,\xi,\eta \in \mathbb{R}$ satisfy the conditions $\rho\tau\geqslant(1+3\rho^2)/2$ and $\rho\eta\geqslant0$. Then,  we have
	\begin{align*}
	\operatorname{Re}\psi_1(i\rho,i\tau,\xi+i\eta)=\frac{\alpha\eta}{\rho}-\beta\rho\tau\leqslant -\frac{\beta(1+3\rho^2)}{2}\leqslant-\frac{\beta}{2},
	\end{align*}
	and this proves that  the function $\psi_1\in \Psi(\Omega)$. For the function $\psi_2$, we have
	\begin{align*}
	\operatorname{Re}\psi_2(i\rho,i\tau,\xi+i\eta)=\frac{\alpha\eta}{\rho}-\beta\rho\tau\leqslant -\frac{\beta(1+3\rho^2)}{2}\leqslant-\frac{\beta}{2},
	\end{align*}and so  the function $\psi_2 \in \Psi(\Omega_2)$.
	Similarly, we have 
	\begin{align*}
	\operatorname{Re}\psi_3(i\rho,i\tau,\xi+i\eta)=\frac{\alpha\eta}{\tau}-\beta\rho\tau\leqslant -\frac{\beta(1+3\rho^2)}{2}\leqslant-\frac{\beta}{2},
	\shortintertext{and }
	\operatorname{Re}\psi_4(i\rho,i\tau, \xi+i\eta)=\frac{\alpha\eta}{\tau}-\beta\rho\tau \leqslant -\frac{\beta(1+3\rho^2)}{2}\leqslant-\frac{\beta}{2},
	\end{align*} so that  the functions $\psi_3\in \Psi(\Omega)$  and  $\psi_4   \in \Psi(\Omega)$.

\end{proof}
\begin{theorem}
	Let $\alpha \leqslant0$ and $\beta >0$. If the function $f \in \mathcal{A}$ satisfy any of the following inequalities
	\begin{enumerate}
		\item[(i)] $\operatorname{Re}\big(\alpha(Q_{SD}/Q_{ST})+\beta Q_{ST} Q_{CV}\big)>-\beta/2$,
		\item[(ii)] $\operatorname{Re}\big(\alpha(Q_{SD}/Q_{CV})+\beta Q_{ST} Q_{CV}\big)>-\beta/2$,	
		\item[(iii)] $\operatorname{Re}\big(\alpha(Q_{SD}/Q_{CV})+\beta Q_{ST} Q_{ST}\big)>0$,
		\item[(iv)] $\operatorname{Re}\big(\alpha(Q_{SD}/ Q_{ST})+\beta Q_{CV} Q_{CV}\big)>0$,
		\item[(v)] $\operatorname{Re}\big(\alpha(Q_{SD}/ Q_{ST})+\beta Q_{ST} Q_{ST}\big)>0$,
		\item[(vi)] $\operatorname{Re}\big(\alpha(Q_{SD}/ Q_{CV})+\beta Q_{CV} Q_{CV}\big)>0$,
		
	\end{enumerate}
	then the function $f$ is starlike.
\end{theorem}
\begin{proof}
	For $i=1,2,\cdots 6$, let $\Omega_i$ be defined by $\Omega_1:=\{w\in \mathbb{C}: \operatorname{Re}w>-\beta/2\}=:\Omega_2$ and $\Omega_3:=\{w \in \mathbb{C}:\operatorname{Re}w>0\}=:\Omega_4:=\Omega_5:=\Omega_6$ and the functions $\psi_i:\mathbb{C}^3\to\mathbb{C}$ be defined  by 
	\begin{alignat*}{5}
	\psi_1(u,v,w)=&\alpha(w/u)+\beta uv,\  
	\psi_2(u,v,w)=\alpha (w/v)+\beta uv,\\
	\psi_3(u,v,w)=&\alpha (w/v)+\beta u^2,\  
	\psi_4(u,v,w)=\alpha (w/u)+\beta v^2,\\
	\psi_5(u,v,w)=&\alpha(w/u)+\beta u^2\quad 
	\text{and}\quad 
	\psi_6(v,w)=\alpha (w/v)+\beta v^2.
	\end{alignat*}
	The hypothesis of the theorem shows that the function $\psi_i$ satisfies  
	\[\psi_i(Q_{ST},Q_{CV},Q_{SD})\in \Omega_i\quad \text{for\ } i=1,2, \cdots 6.\] 
	It then follows from Theorem \ref{sdlemma}  that the function $f$ is starlike provided $\psi_i\in \Psi(\Omega_i)$.
	We complete the proof by showing  that the function $\psi_i\in \Psi(\Omega_i)$.
	
	Let $\alpha \leqslant0$ and $\beta >0$ and  $\rho,\tau,\xi,\eta \in \mathbb{R}$ satisfy the conditions $\rho\tau\geqslant(1+3\rho^2)/2$ and $\rho\eta\geqslant0$. Then,  we have
	\begin{align*}
	\operatorname{Re}\psi_1(i\rho,i\tau,\xi+i\eta)=\frac{\alpha\eta}{\rho}-\beta\rho\tau\leqslant -\frac{\beta(1+3\rho^2)}{2}\leqslant-\frac{\beta}{2} 
	\end{align*}
	and this proves that  the function $\psi_1\in \Psi(\Omega_1)$. For the function $\psi_2$, we have
	\begin{align*}
	\operatorname{Re}\psi_2(i\rho,i\tau,\xi+i\eta)=\frac{\alpha\eta}{\tau}-\beta\rho\tau\leqslant -\frac{\beta(1+3\rho^2)}{2}\leqslant-\frac{\beta}{2},
	\end{align*}and so  the function $\psi_2 \in \Psi(\Omega_2)$.
	Similarly, we have 
	\begin{align*}
	\operatorname{Re}\psi_3(i\rho,i\tau,\xi+i\eta)=\frac{\alpha\eta}{\tau}-\beta\rho^2\leqslant 0,
	\shortintertext{and }
	\operatorname{Re}\psi_4(i\rho,i\tau, \xi+i\eta)=\frac{\alpha\eta}{\rho}-\beta\tau^2 \leqslant 0,
	\end{align*} so that  the functions $\psi_3\in \Psi(\Omega_3)$  and  $\psi_4   \in \Psi(\Omega_4)$. Proceeding in a similar way, we have 
	\begin{align*}
	\operatorname{Re}\psi_5(i\rho,i\tau,\xi+i\eta)=\frac{\alpha\eta}{\rho}-\beta\rho^2 \leqslant 0
	\end{align*}
	and this proves that  the function $\psi_5\in \Psi(\Omega_5)$. For the function $\psi_6$, we have
	\begin{align*}
	\operatorname{Re}\psi_6(i\rho,i\tau,\xi+i\eta)=\frac{\alpha\eta}{\tau}-\beta\tau^2\leqslant 0,
	\end{align*}and so  the function $\psi_6 \in \Psi(\Omega_6)$.

\end{proof}
\begin{theorem}\label{sdtheorem15}
	Let $\alpha>0$ and $\beta \geqslant0$. If the function $f \in \mathcal{A}$ satisfy any of the following inequalities
	\begin{enumerate}
		\item[(i)] $\operatorname{Re}\big(\alpha(Q_{CV}/Q_{ST})-\beta Q_{SD} Q_{CV}\big)<3\alpha/2$,
		\item[(ii)] $\operatorname{Re}\big(\alpha(Q_{CV}/Q_{ST})-\beta Q_{SD} Q_{ST}\big)<3\alpha/2$,
		\item[(iii)] $\operatorname{Re}\big(\alpha(Q_{ST}/Q_{CV})+\beta Q_{SD} Q_{CV}\big)>2\alpha/3$ ,
		\item[(iv)] $\operatorname{Re}\big(\alpha(Q_{ST}/ Q_{CV})+\beta Q_{SD} Q_{ST}\big)>2\alpha/3$ ,
		
	\end{enumerate}
	then the function $f$ is starlike.
\end{theorem}
\begin{proof}
	For $i=1,2,3,4$, let $\Omega_i$ be defined by $\Omega_1:=\{w\in \mathbb{C}: \operatorname{Re}w<3\alpha/2\}=:\Omega_2$ and $\Omega_3:=\{w \in \mathbb{C}:\operatorname{Re}w>2\alpha/3\}=:\Omega_4$ and the functions $\psi_i:\mathbb{C}^3\to\mathbb{C}$ be defined  by 
	\begin{align*}
	\psi_1(u,v,w)&=\alpha(v/u)-\beta vw, \quad 
	\psi_2(u,v,w)=\alpha (v/u)-\beta uw,\\
	\psi_3(u,v,w)&=\alpha (u/v)+\beta vw\quad 
	\text{and}\quad 
	\psi_4(v,w)=\alpha (u/v)+\beta uw.
	\end{align*}
	The hypothesis of the theorem shows that the function $\psi_i$ satisfies  
	\[\psi_i(Q_{ST},Q_{CV},Q_{SD})\in \Omega_i\quad \text{for\ } i=1,2,3,4.\] 
	It then follows from Theorem \ref{sdlemma}  that the function $f$ is starlike provided $\psi_i\in \Psi(\Omega_i)$.
	We complete the proof by showing  that the function $\psi_i\in \Psi(\Omega_i)$.
	
	Let $\alpha>0$ and $\beta \geqslant0$ and  $\rho,\tau,\xi,\eta \in \mathbb{R}$ satisfy the conditions $\rho\tau\geqslant(1+3\rho^2)/2$ and $\rho\eta\geqslant0$. Then,  we have
	\begin{align*}
	\operatorname{Re}\psi_1(i\rho,i\tau,\xi+i\eta)=\frac{\alpha\tau}{\rho}+\beta\tau\eta\geqslant \frac{\alpha(1+3\rho^2)}{2\rho^2}\geqslant\frac{3\alpha}{2} 
	\end{align*}
	and this proves that  the function $\psi_1\in \Psi(\Omega_1)$. For the function $\psi_2$, we have
	\begin{align*}
	\operatorname{Re}\psi_2(i\rho,i\tau,\xi+i\eta)=\frac{\alpha\tau}{\rho}+\beta\rho\eta\geqslant \frac{\alpha(1+3\rho^2)}{2\rho^2}\geqslant\frac{3\alpha}{2},
	\end{align*}and so  the function $\psi_2 \in \Psi(\Omega_2)$.
	The real valued function $2\rho^2/(1+3\rho^2)$ is an increasing function and the maximum value of the function is $2/3$. Then, we have 
	\begin{align*}
	\operatorname{Re}\psi_3(i\rho,i\tau,\xi+i\eta)=\frac{\alpha\rho}{\tau}-\beta\tau\eta\leqslant \frac{2\alpha\rho^2}{1+3\rho^2}\leqslant\frac{2\alpha}{3},
	\shortintertext{and }
	\operatorname{Re}\psi_4(i\rho,i\tau, \xi+i\eta)=\frac{\alpha\rho}{\tau}-\beta\rho\eta \leqslant \frac{2\alpha\rho^2}{1+3\rho^2}\leqslant\frac{2\alpha}{3},
	\end{align*} so that  the functions $\psi_3\in \Psi(\Omega_3)$  and  $\psi_4   \in \Psi(\Omega_4)$.
\end{proof}
\begin{remark}
	Note that substitution of $\alpha=1$ and $\beta=0$ in Part(i) and Part(ii) of Theorem \ref{sdtheorem15} provide the same result as the one obtained in \cite[pp.523]{MR1761956}.
\end{remark}

\end{document}